\documentclass[12pt]{article}
\usepackage{fullpage}
\usepackage{amssymb}
\usepackage{amsmath}
\usepackage{amsthm}
\usepackage{stmaryrd}
\usepackage{xypic}
\usepackage{setspace}
\usepackage{array}
\usepackage{rotating}

\newtheorem{theorem}{Theorem}[section]

\newtheorem{proposition}[theorem]{Proposition}
\newtheorem{corollary}[theorem]{Corollary}

\newcommand{\tref}[1]{Theorem~\textup{\ref{#1}}}
\newcommand{\pref}[1]{Proposition~\textup{\ref{#1}}}
\newcommand{\cref}[1]{Corollary~\textup{\ref{#1}}}

\newcommand{\taref}[1]{Table~\textup{\ref{#1}}}
\newcommand{\eref}[1]{Equation~\textup{\ref{#1}}}

\newcommand{\tx}[1]{\textrm{#1}}
\newcommand{\gp}[2]{\langle #1 \mid #2 \rangle}
\newcommand{\s}{\sigma}
\newcommand{\G}{\Gamma}

\newcommand{\comment}[1]{}

\newcommand{\eps}{\epsilon}

\newcommand{\calK}{\mathcal K}
\newcommand{\calL}{\mathcal L}
\newcommand{\calM}{\mathcal M}
\newcommand{\calP}{\mathcal P}
\newcommand{\calQ}{\mathcal Q}

\newcommand{\GP}{\Gamma^+(\calP)}

\newcommand{\GQ}{\Gamma^+(\calQ)}

\newcommand{\GPQ}{\G^+(\calP \mix \calQ)}
\newcommand{\GcPQ}{\G^+(\calP) \comix \G^+(\calQ)}

\newcommand{\mix}{\diamond}
\newcommand{\comix}{\boxempty}

\begin{document}

\title{Mixing Regular Convex Polytopes}

\author{Gabe Cunningham\\
Department of Mathematics\\
Northeastern University\\
Boston, Massachusetts,  USA, 02115
}

\date{ \today }
\maketitle

\begin{abstract}
The mixing operation for abstract polytopes gives a natural way to construct the minimal common
cover of two polytopes. In this paper, we apply this construction to the regular convex
polytopes, determining when the mix is again a polytope, and 
completely determining the structure of the mix in each case.



\end{abstract}

\section{Introduction}

	Abstract polytopes are combinatorial generalizations of the familiar convex polytopes
	and tessellations of space-forms. Many of the classical geometric constructions carry over
	to the abstract realm. For example, given an abstract polytope $\calP$, we can construct the
	``pyramid'' having $\calP$ as a base, and this construction coincides with the usual pyramid
	construction whenever $\calP$ corresponds to a convex polytope. Other constructions on abstract
	polytopes are new, having no basis in the classical theory. One such example is the mix of
	two polytopes, introduced in \cite{mix-face}; an analogous construction for maps and hypermaps
	appears in \cite{ant2}. 
	The mixing operation is an algebraic construction that finds the minimal natural cover
	of the automorphism group of two regular polytopes. Since there is a standard way to build
	a regular abstract polytope from a group, this construction gives rise to the mix of two polytopes.
	
	By applying the mixing construction to the (abstract versions of the) regular convex polytopes,
	we can find the minimal regular polytopes that cover any subset of the regular convex polytopes.
	Our goal here is to determine their complete structure; how many faces do they have in each rank,
	how many flags are there, and what do their facets and vertex-figures look like? Furthermore,
	we wish to determine which of these new structures are polytopal.
	
	We start by giving some background information on regular abstract polytopes in Section~2. In Section~3,
	we introduce the mixing operation for directly regular polytopes, and we find several criteria
	to determine when the mix of two polytopes is again a polytope. Finally, in Section~4 we
	find the full structure of the mix of any number of regular convex polytopes.

\section{Polytopes}

	General background information on abstract polytopes can be found in \cite[Chs. 2, 3]{arp}.
	Here we review the concepts essential for this paper.

	\subsection{Definition of polytopes}

		Let $\calP$ be a ranked partially ordered set whose elements will be called \emph{faces}. 
		The faces of $\calP$ will range in rank from $-1$ to $n$, and a face of rank $j$ is called a 
		\emph{$j$-face}. The $0$-faces, $1$-faces, and $(n-1)$-faces are also 
		called \emph{vertices}, \emph{edges}, and \emph{facets}, respectively. A \emph{flag} of
		$\calP$ is a maximal chain. We say that two flags are \emph{adjacent} (\emph{$j$-adjacent}) if they differ 
		in exactly one face (their $j$-face, respectively). If $F$ and $G$ are faces of $\calP$
		such that $F \leq G$, then the \emph{section} $G / F$ consists of those faces $H$ such that
		$F \leq H \leq G$.
		
		We say that $\calP$ is an \emph{(abstract) polytope of rank $n$}, also called an \emph{$n$-polytope}, 
		if it satisfies the following four properties:
		
		\begin{enumerate}
		\item There is a unique greatest face $F_n$ of rank $n$ and a unique least face $F_{-1}$ of rank $-1$.
		\item Each flag of $\calP$ has $n+2$ faces.
		\item $\calP$ is \emph{strongly flag-connected}, meaning that if $\Phi$ and $\Psi$
		are two flags of $\calP$, then there is a sequence of flags $\Phi = \Phi_0, \Phi_1, \ldots, \Phi_k = \Psi$ 
		such that for $i = 0, \ldots, k-1$, the flags $\Phi_i$ and $\Phi_{i+1}$ are adjacent, and each
		$\Phi_i$ contains $\Phi \cap \Psi$.
		\item (Diamond condition): Whenever $F < G$, where $F$ is a $(j-1)$-face and $G$ is a $(j+1)$-face
		for some $j$, then there are exactly two $j$-faces $H$ with $F < H < G$.
		\end{enumerate}
		
		Note that due to the diamond condition, any flag $\Phi$ has a unique $j$-adjacent flag (denoted
		$\Phi^j$) for each $j = 0, 1, \ldots, n-1$.
		
		If $F$ is a $j$-face and $G$ is a $k$-face of a polytope with $F \leq G$, then the section $G/F$ is a
		($k-j-1$)-polytope itself. We can identify a face $F$ with the section $F/F_{-1}$; if $F$ is a $j$-face,
		then $F/F_{-1}$ is a $j$-polytope. We call the section $F_n/F$ the \emph{co-face at $F$}. The co-face
		at a vertex is also called a \emph{vertex-figure}. The section $F_{n-1}/F_0$ of a facet over a
		vertex is called a \emph{medial section}. Note that the medial section $F_{n-1}/F_0$ is both a facet
		of the vertex-figure $F_n/F_0$ as well as a vertex-figure of the facet $F_{n-1}/F_{-1}$.

		We sometimes need to work with \emph{pre-polytopes}, which are ranked partially ordered sets that
		satisfy the first, second, and fourth property above, but not necessarily the third. In this paper, all 
		of the pre-polytopes we encounter will be \emph{flag-connected}, meaning that if $\Phi$ and $\Psi$ are two 
		flags, there is a sequence of flags $\Phi = \Phi_0, \Phi_1, \ldots, \Phi_k = \Psi$ such
		that for $i = 0, \ldots, k-1$, the flags $\Phi_i$ and $\Phi_{i+1}$ are adjacent (but we do not require each
		flag to contain $\Phi \cap \Psi$). When working with pre-polytopes, we apply all the same terminology as 
		with polytopes. 
	
	\subsection{Regularity}
	
		For polytopes $\calP$ and $\calQ$, an \emph{isomorphism} from $\calP$ to $\calQ$ is an incidence- and rank-preserving bijection on the set of 
		faces. An isomorphism from $\calP$ to itself is an \emph{automorphism} of $\calP$. We denote the group of 
		all automorphisms of $\calP$ by $\G(\calP)$. There is a natural action of $\G(\calP)$ on the
		flags of $\calP$, and we say that $\calP$ is \emph{regular} if this action is transitive.
		For convex polytopes, this definition is equivalent to any of the usual definitions of regularity.
		
		Given a regular polytope $\calP$, fix a \emph{base flag} $\Phi$. Then the automorphism
		group $\G(\calP)$ is generated by the \emph{abstract reflections} $\rho_0, \ldots, \rho_{n-1}$,
		where $\rho_i$ maps $\Phi$ to the unique flag $\Phi^i$ that is $i$-adjacent to $\Phi$. 
		These generators satisfy $\rho_i^2 = \eps$ for all $i$, and $(\rho_i \rho_j)^2 = \eps$ for all $i$ 
		and $j$ such that $|i - j| \geq 2$.
		We say that $\calP$ has (\emph{Schl\"afli}) \emph{type} $\{p_1,\ldots,p_{n-1}\}$ if
		for each $i = 1, \ldots, n-1$ the order of $\rho_{i-1} \rho_i$ is $p_i$ (with $2 \leq p_i \leq \infty$).
		We also use $\{p_1, \ldots, p_{n-1}\}$ to represent the universal regular polytope of this type,
		which has an automorphism group with no relations other than those mentioned above.
		We denote the (Coxeter) group $\G(\{p_1, \ldots, p_{n-1}\})$ by $[p_1, \ldots, p_{n-1}]$.
		Whenever this universal polytope corresponds to a regular convex polytope, then the name used
		here is the same as the usual Schl\"afli symbol for that polytope (see \cite{coxeter}).
		
		For $I \subseteq \{0, 1, \ldots, n-1\}$ and a group $\G = \langle \rho_0, \ldots, \rho_{n-1} \rangle$,
		we define $\G_I := \langle \rho_i \mid i \in I \rangle$. 
		The strong flag-connectivity of polytopes induces the following {\em intersection property\/} in the group:
		\begin{equation}
		\label{eq:reg-int}
		\G_I \cap \G_J = \G_{I \cap J} 
		\;\; \textrm{ for } I,J \subseteq \{0, \ldots, n-1\}.
		\end{equation}

		In general, if $\G = \langle \rho_0, \ldots, \rho_{n-1} \rangle$ is a group such that each
		$\rho_i$ has order $2$ and such that $(\rho_i \rho_j)^2 = \eps$ whenever $|i - j| \geq 2$, then
		we say that $\G$ is a \emph{string group generated by involutions} (or \emph{sggi}). If
		$\G$ also satisfies the intersection property given above, then we call $\G$ a \emph{string
		C-group}. There is a natural way of building a regular polytope $\calP(\G)$ from a string
		C-group $\G$ such that $\G(\calP(\G)) = \G$ (see \cite[Ch. 2E]{arp}). Therefore, we get a one-to-one 
		correspondence between regular $n$-polytopes and string C-groups on $n$ specified generators.
		
		Given a regular polytope $\calP$ with automorphism group $\G(\calP) = \langle \rho_0, \ldots, \rho_{n-1}
		\rangle$, we define the \emph{abstract rotations} $\s_i := \rho_{i-1} \rho_i$ for $i=1, \ldots, n$. 
		These elements generate the \emph{rotation subgroup} $\GP$ of $\G(\calP)$, which has index at 
		most~$2$. We say that $\calP$ is \emph{directly regular} if this index is $2$. Note that the regular
		convex polytopes are all directly regular, and that any section of a directly regular
		polytope is directly regular.
	
		The rotation subgroup of a directly regular polytope satisfies the relations 
		\begin{equation}
		\label{eq:tau}
		(\s_i \cdots \s_j)^2 = \eps \tx{ for $i < j$.}
		\end{equation}
		It also satisfies an intersection property analogous to that for the automorphism groups of
		regular polytopes. For $1 \leq i < j \leq n-1$ define $\tau_{i,j}:= \s_i \cdots \s_j$.
		By convention, we also define $\tau_{i,i} = \s_i$, and for $0 \leq i \leq n$, we define $\tau_{0,i} =
		\tau_{i,n} = \eps$.
		For $I \subseteq \{0, \ldots, n-1\}$ and $\G^+ := \GP$, set
		\[ \G^+_I := \langle \tau_{i,j} \mid i \leq j \tx{ and } i-1, j \in I \rangle.  \]
		Then the \emph{intersection property} for $\G^+$ is given by:
		\begin{equation}
		\label{eq:chiral-int}
		\G^+_I \cap \G^+_J = \G^+_{I\cap J}  
		\;\; \textrm{ for } I,J \subseteq \{0, \ldots, n-1\}.
		\end{equation}
	
		Let $\calP$ and $\calQ$ be two polytopes (or flag-connected pre-polytopes) of the same rank, not 
		necessarily regular. A function $\gamma: \calP \to \calQ$ is called a
		\emph{covering} if it preserves incidence of faces, ranks of faces, and adjacency of flags; then $\gamma$ is
		necessarily surjective, by the flag-connectedness of $\calQ$. We say that $\calP$ \emph{covers} $\calQ$
		if there exists a covering $\gamma: \calP \to \calQ$.

		If $\calP$ and $\calQ$ are directly regular, then their rotation groups
		are both quotients of 
		\[ W^+ := \langle \s_1, \ldots, \s_{n-1} \mid (\s_i \cdots \s_j)^2 = \eps \tx{ for $1 \leq i < j \leq n-1$}
		\rangle. \]
		Therefore there are normal subgroups $M$ and $K$ of $W^+$ such that $\GP = W^+/M$ and $\GQ = W^+/K$. Then
		$\calP$ covers $\calQ$ if and only if $M \leq K$.
		
\section{Mixing polytopes}

	In this section, we will define the mix of two finitely presented groups, which naturally
	gives rise to a way to mix polytopes. The mixing operation is analogous to the join of hypermaps 
	\cite{ant2} and the parallel product of maps \cite{wilson}.
	
	Let $\G = \langle x_1, \ldots, x_n \rangle$ and $\G' =
	\langle x_1', \ldots, x_n' \rangle$ be groups with $n$ specified generators. Then the elements
	$z_i = (x_i, x_i') \in \G \times \G'$ (for $i = 1, \ldots, n$) generate a subgroup of
	$\G \times \G'$ that we call the \emph{mix} of $\G$ and $\G'$ and denote $\G \mix \G'$
	(see \cite[Ch.7A]{arp}).

	If $\calP$ and $\calQ$ are directly regular $n$-polytopes, we can mix their automorphism groups or
	their rotation groups. The theory is essentially the same in either case, but it ends up being easier
	to use their rotation groups. Let $\GP = \langle \s_1, \ldots, \s_{n-1} \rangle$ and $\GQ = \langle
	\s_1', \ldots, \s_{n-1}' \rangle$. Let $\beta_i = (\s_i, \s_i')$ for $i = 1, \ldots, n-1$.
	Then $\GP \mix \GQ = \langle \beta_1, \ldots, \beta_{n-1} \rangle$. We note that for $i < j$, we have
	$(\beta_i \cdots \beta_j)^2 = \eps$, so that the group $\GP \mix \GQ$ satisfies Equation~\ref{eq:tau}.
	In general, however, it will not have the intersection property (\eref{eq:chiral-int}) with respect to its
	generators $\beta_1, \ldots, \beta_{n-1}$. Nevertheless, it is possible to build a directly regular poset from
	$\GP \mix \GQ$ using the method outlined in \cite{chiral}, and we denote that poset $\calP \mix \calQ$
	and call it the \emph{mix} of $\calP$ and $\calQ$. (In fact, this poset is always a flag-connected
	pre-polytope.) Thus $\GPQ = \GP \mix \GQ$. If $\GP \mix \GQ$ satisfies
	the intersection property, then $\calP \mix \calQ$ is in fact a polytope.
	
	If the facets of $\calP$ are $\calK$ and the facets of $\calQ$ are $\calK'$, then the facets of
	$\calP \mix \calQ$ are $\calK \mix \calK'$. The vertex-figures of $\calP \mix \calQ$ are obtained
	similarly.

	The following proposition is proved in \cite{const}:

	\begin{proposition}
	\label{prop:mix}
	Let $\calP$ and $\calQ$ be directly regular polytopes with $\GP = W^+/M$ and
	$\GQ = W^+/K$. Then $\GPQ \simeq W^+/(M \cap K)$.
	\end{proposition}
	
	In most of the cases we will encounter in this paper, $\GP \mix \GQ$ is in fact equal to $\GP \times \GQ$.
	In order to determine when that happens, it is useful to introduce the \emph{comix} of two groups.
	If $\G$ has presentation $\gp{x_1, \ldots, x_n}{R}$ and $\G'$ has presentation $\gp{x_1', \ldots, x_n'}{S}$,
	then we define the comix of $\G$ and $\G'$, denoted $\G \comix \G'$, to be the group with presentation
	\[ \gp{x_1, x_1', \ldots, x_n, x_n'}{R, S, x_1^{-1}x_1', \ldots, x_n^{-1}x_n'}.\]
	Informally speaking, we can just add the relations from $\G'$ to $\G$, rewriting them to use
	$x_i$ in place of $x_i'$. 

	Just as the mix of two rotation groups has a simple description in terms of quotients of $W^+$, so
	does the comix of two rotation groups.

	\begin{proposition}
	\label{prop:comix}
	Let $\calP$ and $\calQ$ be directly regular polytopes with $\GP = W^+/M$ and
	$\GQ = W^+/K$. Then $\GcPQ \simeq W^+/MK$.
	\end{proposition}

	\begin{proof}
	Let $\GP = \langle \s_1, \ldots, \s_{n-1} \mid R \rangle$, and let $\GQ = \langle \s_1, \ldots, \s_{n-1} \mid
	S \rangle$, where $R$ and $S$ are sets of relators in $W^+$. 
	Then $M$ is the normal closure of $R$ in $W^+$ and $K$ is the normal closure of $S$ in $W^+$.
	We can write $\GcPQ = \langle \s_1, \ldots, \s_{n-1} \mid R \cup S \rangle$, so we want to show that $MK$ is
	the normal closure of $R \cup S$ in $W^+$. It is clear that $MK$ contains $R \cup S$, and since
	$M$ and $K$ are normal, $MK$ is normal, and so it contains the normal closure of $R \cup S$.
	To show that $MK$ is contained in the normal closure of $R \cup S$, it suffices to show that if
	$N$ is a normal subgroup of $W^+$ that contains $R \cup S$, then it must also contain $MK$. Clearly,
	such an $N$ must contain the normal closure $M$ of $R$ and the normal closure $K$ of $S$. Therefore,
	$N$ contains $MK$, as desired.
	\end{proof}

	\subsection{Size of the Mix}

		Now we can determine how the size of $\GP \mix \GQ$ is related to the size of $\GP \comix \GQ$.

		\begin{proposition}
		\label{prop:SizeOfMix}
		Let $\calP$ and $\calQ$ be finite directly regular $n$-polytopes. Then 
		\[ |\GP \mix \GQ| \cdot |\GcPQ| = |\GP| \cdot |\GQ|. \]
		\end{proposition}

		\begin{proof}
		Let $\GP = W^+/M$ and $\GQ = W^+/K$. Then by \pref{prop:mix}, $\GP \mix \GQ = W^+/(M \cap K)$, and 
		by \pref{prop:comix}, $\GcPQ = W^+/MK$.
		Let $\pi_1: \GP \mix \GQ \to \GP$ and $\pi_2: \GQ \to \GcPQ$ be the natural epimorphisms. Then
		$\ker \pi_1 \simeq M/(M \cap K)$ and $\ker \pi_2 \simeq MK/K \simeq M/(M \cap K)$. Therefore,
		we have that 
		\[|\GP \mix \GQ| = |\GP||\ker \pi_1| = |\GP||\ker \pi_2| = |\GP||\GQ|/|\GcPQ|,\]
		and the result follows.
		\end{proof}
		
		The following corollary is immediate.
		
		\begin{corollary}
		\label{cor:flags-mix}
		Let $\calP$ and $\calQ$ be finite directly regular $n$-polytopes such that $\GP \comix \GQ$
		is trivial. Then $\GP \mix \GQ = \GP \times \GQ$. Furthermore, if $\calP$ has $g$ flags and
		$\calQ$ has $h$ flags, then $\calP \mix \calQ$ has $gh/2$ flags.
		\end{corollary}

		We conclude this section by determining some cases for which $\GP \comix \GQ$ is indeed trivial.
		
		\begin{proposition}
		\label{prop:trivial}
		Let $\calP$ and $\calQ$ be directly regular $n$-polytopes. Let $\GP \comix \GQ = \langle \s_1, \ldots,
		\s_{n-1} \rangle$. Suppose $\s_i$ is trivial for some $i$. If $i \geq 2$, then $\s_{i-1}$ 
		has order $1$ or $2$, and if $i \leq n-2$, then $\s_{i+1}$ has order $1$ or $2$.
		\end{proposition}
		
		\begin{proof}
		If $\s_i = \eps$ for $i \geq 2$, then the relation $(\s_{i-1} \s_i)^2 = \eps$
		reduces to $\s_{i-1}^2 = \eps$, and thus $\s_{i-1}$ has order $1$ or $2$. The proof for $\s_{i+1}$ is the same.
		\end{proof}
		
		\begin{corollary}
		\label{cor:gcd-comix}
		Let $\calP$ be a directly regular polytope of type $\{p_1, \ldots, p_n\}$ and let $\calQ$ be a 
		directly regular polytope
		of type $\{q_1, \ldots, q_n\}$. Suppose that $\gcd(p_i, q_i)$ is odd for all $i$, and that
		it is $1$ for at least one $i$. Then $\GP \comix \GQ$ is trivial, and thus $\GP \mix \GQ = \GP \times \GQ$.
		\end{corollary}
		
		\begin{proof}
		Fix a $k$ such that $\gcd(p_k, q_k) = 1$. Then $\s_k$ is trivial in $\GP \comix \GQ$.
		Now, by \pref{prop:trivial}, if $k \geq 2$, then $\s_{k-1}$ has
		order $1$ or $2$. On the other hand, $\s_{k-1}$ also has order dividing $\gcd(p_{k-1}, q_{k-1})$,
		which is odd. Therefore, $\s_{k-1}$ is trivial. Proceeding in this manner,
		we conclude that all of the generators $\s_i$ for $i < k$ are trivial. Similarly, if $k \leq n-2$,
		then $\s_{k+1}$ must be trivial by the same reasoning, and we find that all the generators
		are trivial.
		\end{proof}
		
	\subsection{Polytopality of the Mix}
	
		In order for the mix of $\calP$ and $\calQ$ to be a polytope, we need for the group 
		$\GP \mix \GQ$ to satisfy the intersection property (\eref{eq:chiral-int}). We naturally would like to
		have simple conditions that determine when this is the case. 
		We start with two results from \cite{const}.
		
		\begin{proposition}
		\label{prop:facets-cover}
		Let $\calP$ be a directly regular $n$-polytope with facets isomorphic to $\calK$.
		Let $\calQ$ be a directly regular flag-connected $n$-pre-polytope with facets isomorphic to $\calK'$.
		If $\calK$ covers $\calK'$ or $\calK'$ covers $\calK$, then $\calP \mix \calQ$ is polytopal.
		\end{proposition}

		\begin{proposition}
		\label{prop:rel-prime-type}
		Let $\calP$ be a directly regular $n$-polytope of type $\{p_1, \ldots, p_{n-1}\}$, and let
		$\calQ$ be a directly regular $n$-polytope of type $\{q_1, \ldots, q_{n-1}\}$.
		If $p_i$ and $q_i$ are relatively prime for each $i = 1, \ldots, n-1$, then $\calP \mix \calQ$ is a directly regular $n$-polytope of type $\{p_1 q_1, \ldots, p_{n-1} q_{n-1}\}$, and $\GPQ = \GP \times \GQ$.
		\end{proposition}
		
		In general, when we mix $\calP$ and $\calQ$, we have to verify the full intersection property. But as we 
		shall see, some parts of the intersection property are automatic. Recall that for a subset $I$ of 
		$\{0, \ldots, n-1\}$ and a 
		rotation group $\G^+ = \langle \s_1, \ldots, \s_{n-1} \rangle$, we define 
		\[ \G^+_I = \langle \tau_{i,j} \mid i \leq j \tx{ and } i-1,j \in I \rangle, \]
		where $\tau_{i,j} = \s_i \cdots \s_j$.

		\begin{proposition}
		\label{prop:disjoint-int-prop}
		Let $\calP$ and $\calQ$ be directly regular $n$-polytopes, and let $I, J \subseteq \{0, \ldots, n-1\}$. 
		Let $\Lambda = \GP$, $\Delta = \GQ$, and $\G^+ = \Lambda \mix \Delta$.
		Then $\G^+_I \cap \G^+_J \leq \Lambda_{I \cap J} \times \Delta_{I \cap J}$.
		Furthermore, if $\G^+_I = \Lambda_I \times \Delta_I$ and $\G^+_J = \Lambda_J \times \Delta_J$,
		then $\G^+_I \cap \G^+_J = \Lambda_{I \cap J} \times \Delta_{I \cap J}$.
		\end{proposition}
		
		\begin{proof}
		Since $\G^+_I \leq \Lambda_I \times \Delta_I$ and $\G^+_J \leq \Lambda_J \times \Delta_J$, we have
		\begin{align*}
		\G^+_I \cap \G^+_J & \leq (\Lambda_I \times \Delta_I) \cap (\Lambda_J \times \Delta_J) \\
		& = (\Lambda_I \cap \Lambda_J) \times (\Delta_I \cap \Delta_J) \\
		& = \Lambda_{I \cap J} \times \Delta_{I \cap J},
		\end{align*}
		where the last line follows from the polytopality of $\calP$ and $\calQ$. This proves the first part.
		For the second part, we note that if $\G^+_I = \Lambda_I \times \Delta_I$ and $\G^+_J = \Lambda_J \times 
		\Delta_J$, then we get equality in the first line.
		\end{proof}
		
		\begin{corollary}
		\label{cor:disjoint-int-prop}
		Let $\calP$ and $\calQ$ be directly regular $n$-polytopes, and
		let $\GP \mix \GQ = \langle \beta_1, \ldots, \beta_{n-1} \rangle.$ 
		Let $1 \leq i, j \leq n-1$. Then
		\[ \langle \beta_1, \ldots, \beta_i \rangle \cap \langle \beta_j, \ldots, \beta_{n-1} \rangle
			\leq \langle \beta_j, \ldots, \beta_i \rangle. \]
		In particular, if $j > i$ then the given intersection is trivial.
		\end{corollary}
		
		\begin{proof}
		The claim follows directly from \pref{prop:disjoint-int-prop} by taking $I = \{0, \ldots, i\}$ and
		$J = \{j-1, \ldots, n-1\}$.
		\end{proof}
		
		\begin{corollary}
		\label{cor:polyhedra}
		Let $\calP$ and $\calQ$ be directly regular polyhedra. Then $\calP \mix \calQ$ is a 
		directly regular polyhedron.
		\end{corollary}

		\begin{proof}
		In order for $\calP \mix \calQ$ to be a polyhedron (and not just a pre-polyhedron), it must satisfy the 
		intersection property. For polyhedra, the only requirement is that $\langle \beta_1 \rangle \cap 
		\langle \beta_2 \rangle = \langle \eps \rangle$, which holds by \cref{cor:disjoint-int-prop}.
		\end{proof}
		
		\cref{cor:polyhedra} is extremely useful. In addition to telling us that the mix of any two
		polyhedra is a polyhedron, it makes it simpler to verify the polytopality of the mix of $4$-polytopes,
		since the facets and vertex-figures of the mix are guaranteed to be polytopal.

		We now prove some general results that work for polytopes in any rank.
		
		\begin{proposition}
		\label{prop:medial-sections}
		Let $\calP$ be a directly regular $n$-polytope of type $\{\calK, \calL\}$ with medial sections $\calM$, and
		let $\calQ$ be a directly regular $n$-polytope of type $\{\calK', \calL'\}$ with medial sections $\calM'$.
		Suppose that $\calK \mix \calK'$ and $\calL \mix \calL'$ are polytopal. If
		$\G^+(\calM \mix \calM') = \G^+(\calM) \times \G^+(\calM')$, then $\calP \mix \calQ$ is polytopal.
		\end{proposition}
		
		\begin{proof}
		Let $\GP = \langle \s_1, \ldots, \s_{n-1} \rangle$ and let $\GQ = \langle \s_1', \ldots, \s_{n-1}' 
		\rangle$. Let $\beta_i = (\s_i, \s_i')$, so that $\GPQ = \langle \beta_1, \ldots, \beta_{n-1} \rangle$.
		The facets of $\calP \mix \calQ$ are $\calK \mix \calK'$, and the vertex figures are $\calL \mix \calL'$,
		both of which are polytopal. Thus, to verify the polytopality of $\calP \mix \calQ$, it suffices to
		check that $\langle \beta_1, \ldots, \beta_{n-2} \rangle \cap \langle \beta_2, \ldots, \beta_{n-1} 
		\rangle = \langle \beta_2, \ldots, \beta_{n-2} \rangle$ \cite{chiral}. From \pref{prop:disjoint-int-prop}
		we get that 
		\[ \langle \beta_1, \ldots, \beta_{n-2} \rangle \cap \langle \beta_2, \ldots, \beta_{n-1} 
		\rangle \leq \langle \s_2, \ldots, \s_{n-2} \rangle \times \langle \s_2', \ldots, \s_{n-2}' \rangle. \]
		The right hand side is just $\G^+(\calM) \times \G^+(\calM')$, and since this is equal to
		$\G^+(\calM \mix \calM') = \langle \beta_2, \ldots, \beta_{n-2} \rangle$, the result follows.
		\end{proof}

		\begin{theorem}
		\label{thm:medial-dir-prod}
		Let $\calP$ be a directly regular $n$-polytope of type $\{\calK, \calL\}$ with medial sections $\calM$, and
		let $\calQ$ be a directly regular $n$-polytope of type $\{\calK', \calL'\}$ with medial sections $\calM'$.
		Suppose that $\calK \mix \calK'$ and $\calL \mix \calL'$ are polytopal, and suppose
		that $\G^+(\calK \mix \calK') = \G^+(\calK) \times \G^+(\calK')$ and
		that $\G^+(\calL \mix \calL') = \G^+(\calL) \times \G^+(\calL')$. Then
		$\calP \mix \calQ$ is polytopal if and only if $\G^+(\calM \mix \calM') =
		\G^+(\calM) \times \G^+(\calM')$.
		\end{theorem}
		
		\begin{proof}
		Let $\GP = \langle \s_1, \ldots, \s_{n-1} \rangle$ and let $\GQ = \langle \s_1', \ldots, \s_{n-1}' 
		\rangle$. Let $\beta_i = (\s_i, \s_i')$, so that $\GPQ = \langle \beta_1, \ldots, \beta_{n-1} \rangle$.
		\pref{prop:medial-sections} proves that if $\G^+(\calM \mix \calM') = \G^+(\calM) \times \G^+(\calM')$,
		then $\calP \mix \calQ$ is polytopal. Conversely, suppose $\calP \mix \calQ$ is polytopal.
		Since $\calP \mix \calQ$ is polytopal, we have that
		\[ \langle \beta_2, \ldots, \beta_{n-2} \rangle = \langle \beta_1, \ldots, \beta_{n-2} \rangle
		\cap \langle \beta_2, \ldots, \beta_{n-1} \rangle. \]
		Now, the right-hand side is $\G^+(\calK \mix \calK') \cap \G^+(\calL \mix \calL')$. Since
		$\G^+(\calK \mix \calK') = \G^+(\calK) \times \G^+(\calK')$ and
		$\G^+(\calL \mix \calL') = \G^+(\calL) \times \G^+(\calL')$, we can apply 
		\pref{prop:disjoint-int-prop} to see that the right hand side is equal to $\langle \s_2, \ldots, \s_{n-2} 
		\rangle \times \langle \s_2', \ldots \s_{n-2}' \rangle$, which is $\G^+(\calM) \times \G^+(\calM')$.
		Since the left-hand side is equal to $\G^+(\calM \mix \calM')$, we get that $\G^+(\calM \mix \calM') =
		\G^+(\calM) \times \G^+(\calM')$, as desired.
		\end{proof}

\section{The Mix of the Regular Convex Polytopes}

	Now we will actually mix the convex polytopes. In each case, we will determine the number of flags,
	the number of faces of each rank, and whether the mix is polytopal. The Schl\"afli type of the mix
	is easily obtained by taking the least common multiple of the corresponding entries in all the component
	polytopes. In most cases, the results have been found in two ways: by judicious use of the preceding results, 
	and by direct calculation using GAP \cite{gap}. 

	\subsection{Rank 3}
	
		The five regular convex polyhedra are the tetrahedron $\{3, 3\}$, the octahedron $\{3, 4\}$, the
		icosahedron $\{3, 5\}$, the cube $\{4, 3\}$, and the dodecahedron $\{5, 3\}$. By \cref{cor:polyhedra},
		the mix of any number of these is a polyhedron (i.e., polytopal). For the mix of two regular convex
		polyhedra $\calP$ and $\calQ$, \cref{cor:gcd-comix} shows that in every case, we get
		$\GP \mix \GQ = \GP \times \GQ$. By carefully grouping polytopes, we can use \cref{cor:gcd-comix}
		for the mix of three or more regular convex polyhedra as well. 
		For example, since $\{3, 3\} \mix \{3, 4\}$ is of type $\{3, 12\}$, we can apply
		\cref{cor:gcd-comix} to $(\{3, 3\} \mix \{3, 4\}) \mix \{3, 5\}$.
		
		Due to the above considerations, we can always apply \cref{cor:flags-mix} to find the number
		of flags in the mix. To find the remaining information, we note that if $\calP$ is a regular polyhedron of 
		type $\{p, q\}$ with $g$ flags, then $\calP$ has $g/(2q)$ vertices, $g/4$ edges, and $g/(2p)$ facets.

		Information about the mix of the regular convex polyhedra is summarized in \taref{tab:polyhedra}, where
		$f_0$ is the number of vertices, $f_1$ is the number of edges, $f_2$ is the number of $2$-faces,
		and $g$ is the size of the automorphism group (which is also the number of flags).
		Since $\calP \mix \calP = \calP$ for any polytope $\calP$, there are only finitely many mixes
		of the regular convex polyhedra.
		Note that whenever one of the mixes has the same number of vertices and facets, the
		polytope is in fact self-dual.
		
		\begin{table}[htbp]
		\centering
		\begin{tabular}{c c c c c}
		Polyhedron & $f_0$ & $f_1$ & $f_2$ & $g$ \\ \hline
		$\{3, 3\} \mix \{3, 4\}$ & 24  & 144  & 96  & $576$ \\ \hline
		$\{3, 3\} \mix \{3, 5\}$ & 48  & 360  & 240 & $1440$ \\ \hline
		$\{3, 3\} \mix \{4, 3\}$ & 96  & 144  & 24  & $576$ \\ \hline
		$\{3, 3\} \mix \{5, 3\}$ & 240 & 360  & 48  & $1440$ \\ \hline
		$\{3, 4\} \mix \{3, 5\}$ & 72  & 720  & 480 & $2880$ \\ \hline
		$\{3, 4\} \mix \{4, 3\}$ & 48  & 288  & 48  & $1152$ \\ \hline
		$\{3, 4\} \mix \{5, 3\}$ & 120 & 720  & 96  & $2880$ \\ \hline
		$\{3, 5\} \mix \{4, 3\}$ & 96  & 720  & 120 & $2880$ \\ \hline
		$\{3, 5\} \mix \{5, 3\}$ & 240 & 1800 & 240 & $7200$ \\ \hline
		$\{4, 3\} \mix \{5, 3\}$ & 480 & 720  & 72  & $2880$ \\ \hline
		$\{3, 3\} \mix \{3, 4\} \mix \{3, 5\}$ & 288  & 8640  & 5760 & $34560$ \\ \hline
		$\{3, 3\} \mix \{3, 4\} \mix \{4, 3\}$ & 576  & 3456  & 576  & $13824$ \\ \hline
		$\{3, 3\} \mix \{3, 4\} \mix \{5, 3\}$ & 1440 & 8640  & 1152 & $34560$ \\ \hline
		$\{3, 3\} \mix \{3, 5\} \mix \{4, 3\}$ & 1152 & 8640  & 1440 & $34560$ \\ \hline
		$\{3, 3\} \mix \{3, 5\} \mix \{5, 3\}$ & 2880 & 21600 & 2880 & $86400$ \\ \hline
		$\{3, 3\} \mix \{4, 3\} \mix \{5, 3\}$ & 5760 & 8640  & 288  & $34560$ \\ \hline
		$\{3, 4\} \mix \{3, 5\} \mix \{4, 3\}$ & 576  & 17280 & 2880 & $69120$ \\ \hline
		$\{3, 4\} \mix \{3, 5\} \mix \{5, 3\}$ & 1440 & 43200 & 5760 & $172800$ \\ \hline
		$\{3, 4\} \mix \{4, 3\} \mix \{5, 3\}$ & 2880 & 17280 & 576  & $69120$ \\ \hline
		$\{3, 5\} \mix \{4, 3\} \mix \{5, 3\}$ & 5760 & 43200 & 1440 & $172800$ \\ \hline
		$\{3, 3\} \mix \{3, 4\} \mix \{3, 5\} \mix \{4, 3\}$ & 6912  & 207360  & 34560 & $829440$  \\ \hline
		$\{3, 3\} \mix \{3, 4\} \mix \{3, 5\} \mix \{5, 3\}$ & 17280 & 518400  & 69120 & $2073600$  \\ \hline
		$\{3, 3\} \mix \{3, 4\} \mix \{4, 3\} \mix \{5, 3\}$ & 34560 & 207360  & 6912  & $829440$  \\ \hline
		$\{3, 3\} \mix \{3, 5\} \mix \{4, 3\} \mix \{5, 3\}$ & 69120 & 518400  & 17280 & $2073600$  \\ \hline
		$\{3, 4\} \mix \{3, 5\} \mix \{4, 3\} \mix \{5, 3\}$ & 34560 & 1036800 & 34560 & $4147200$  \\ \hline
		$\{3, 3\} \mix \{3, 4\} \mix \{3, 5\} \mix \{4, 3\} \mix \{5, 3\}$ & 414720
		& 12441600 & 414720 & 49766400  \\ \hline
		\end{tabular}
		\caption{The mix of regular convex polyhedra}
		\label{tab:polyhedra}
		\end{table}

	\subsection{Rank 4}

		Next we present the mix of the regular convex 4-polytopes. There are six regular convex polytopes:
		$\calP_1 = \{3, 3, 3\}$, $\calP_2
		= \{3, 3, 4\}$, $\calP_3 = \{3, 3, 5\}$, $\calP_4 = \{3, 4, 3\}$, $\calP_5 = \{4, 3, 3\}$, and
		$\calP_6 = \{5, 3, 3\}$. For $I = \{i_1, \ldots, i_k \} \subseteq \{1, 2, \ldots, 6\}$ we define 
		$\calP_I = \calP_{i_1} \mix \cdots \mix \calP_{i_k}$. As was the case in rank $3$, the mix
		of the rotation groups always turns out to be the direct product of the components. Finding
		the number of flags of the mix is then simple. To find the number $f_3$ of cells (i.e. facets) of
		the mix, we divide the number of flags by the number of flags in the facet, which we find by
		using \taref{tab:polyhedra}. Similar calculations give the number of vertices, edges, and $2$-faces.
		
		Now we address the concern of polytopality. Unlike the mix of regular convex polyhedra, there
		are mixes of regular convex $4$-polytopes that are not polytopal. 
		For the mix of two regular convex $4$-polytopes, we
		can appeal directly to \pref{prop:facets-cover} or \tref{thm:medial-dir-prod}, and this settles
		every case. For the mix of three or more regular convex $4$-polytopes, we can always group them
		in such a way as to again apply one of these results. For example, the vertex-figures of $\{3, 3, 3\}
		\mix \{3, 3, 4\}$ are $\{3, 3\} \mix \{3, 4\}$, and so they cover $\{3, 3\}$. Therefore, the mix
		$\{3, 3, 3\} \mix \{3, 3, 4\} \mix \{4, 3, 3\}$ is polytopal by \pref{prop:facets-cover}.
		Tables \ref{tab:4-poly-1} and \ref{tab:4-poly-2} summarize our results.
		
		\begin{sidewaystable}[htbp]
		\centering
		\begin{tabular}{c c c c c c c}
		$I$ & $f_0$ & $f_1$ & $f_2$ & $f_3$ & $g$ & Polytopal?\\ \hline
		$\{1, 2\}$ & 40     & 480     & 1920    & 960    & $23040$     & Y\\ \hline
		$\{1, 3\}$ & 600    & 14400   & 72000   & 36000  & $864000$    & Y\\ \hline
		$\{1, 4\}$ & 120    & 5760    & 5760    & 120    & $69120$     & Y\\ \hline
		$\{1, 5\}$ & 960    & 1920    & 480     & 40     & $23040$     & Y\\ \hline
		$\{1, 6\}$ & 36000  & 72000   & 14400   & 600    & $864000$    & Y\\ \hline
		$\{2, 3\}$ & 960    & 34560   & 230400  & 115200 & $2764800$   & Y\\ \hline
		$\{2, 4\}$ & 192    & 4608    & 18432   & 384    & $221184$    & Y\\ \hline
		$\{2, 5\}$ & 128    & 1536    & 1536    & 128    & $73728$     & N\\ \hline
		$\{2, 6\}$ & 4800   & 57600   & 46080   & 1920   & $2764800$   & N\\ \hline
		$\{3, 4\}$ & 2880   & 138240  & 691200  & 14400  & $8294400$   & Y\\ \hline
		$\{3, 5\}$ & 1920   & 46080   & 57600   & 4800   & $2764800$   & N\\ \hline
		$\{3, 6\}$ & 72000  & 1728000 & 1728000 & 72000  & $103680000$ & N\\ \hline
		$\{4, 5\}$ & 384    & 18432   & 4608    & 192    & $221184$    & Y\\ \hline
		$\{4, 6\}$ & 14400  & 691200  & 138240  & 2880   & $8294400$   & Y\\ \hline
		$\{5, 6\}$ & 115200 & 230400  & 34560   & 960    & $2764800$   & Y\\ \hline
		$\{1, 2, 3\}$ & 4800     & 691200    & 13824000  & 6912000  & $165888000$   & Y\\ \hline
		$\{1, 2, 4\}$ & 960      & 276480    & 1105920   & 23040    & $13271040$    & Y\\ \hline
		$\{1, 2, 5\}$ & 7680     & 92160     & 92160     & 7680     & $4423680$     & Y\\ \hline
		$\{1, 2, 6\}$ & 288000   & 3456000   & 2764800   & 115200   & $165888000$   & Y\\ \hline
		$\{1, 3, 4\}$ & 14400    & 8294400   & 41472000  & 864000   & $497664000$   & Y\\ \hline
		$\{1, 3, 5\}$ & 115200   & 2764800   & 3456000   & 288000   & $165888000$   & Y\\ \hline
		$\{1, 3, 6\}$ & 4320000  & 103680000 & 103680000 & 4320000  & $6220800000$  & Y\\ \hline
		$\{1, 4, 5\}$ & 23040    & 1105920   & 276480    & 960      & $13271040$    & Y\\ \hline
		$\{1, 4, 6\}$ & 864000   & 41472000  & 8294400   & 14400    & $497664000$   & Y\\ \hline
		$\{1, 5, 6\}$ & 6912000  & 13824000  & 691200    & 4800     & $165888000$   & Y\\ \hline
		$\{2, 3, 4\}$ & 23040    & 6635520   & 132710400 & 2764800  & $1592524800$  & Y\\ \hline
		$\{2, 3, 5\}$ & 15360    & 2211840   & 11059200  & 921600   & $530841600$   & N\\ \hline
		$\{2, 3, 6\}$ & 576000   & 82944000  & 331776000 & 13824000 & $19906560000$ & N\\ \hline
		$\{2, 4, 5\}$ & 3072     & 884736    & 884736    & 3072     & $42467328$    & N\\ \hline
		$\{2, 4, 6\}$ & 115200   & 33177600  & 26542080  & 46080    & $1592524800$  & N\\ \hline
		$\{2, 5, 6\}$ & 921600   & 11059200  & 2211840   & 15360    & $530841600$   & N\\ \hline
		$\{3, 4, 5\}$ & 46080    & 26542080  & 33177600  & 115200   & $1592524800$  & N\\ \hline
		$\{3, 4, 6\}$ & 1728000  & 995328000 & 995328000 & 1728000  & $59719680000$ & N\\ \hline
		$\{3, 5, 6\}$ & 13824000 & 331776000 & 82944000  & 576000   & $19906560000$ & N\\ \hline
		$\{4, 5, 6\}$ & 2764800  & 132710400 & 6635520   & 23040    & $1592524800$  & Y\\ \hline
		\end{tabular}
		\caption{The mix of 2 or 3 regular convex 4-polytopes}
		\label{tab:4-poly-1}
		\end{sidewaystable}

		\begin{sidewaystable}[htbp]
		\centering
		\begin{tabular}{c c c c c c c}
		$I$ & $f_0$ & $f_1$ & $f_2$ & $f_3$ & $g$ & Polytopal?\\ \hline
		$\{1, 2, 3, 4\}$ & 115200    & 398131200    & 7962624000   & 165888000 & $95551488000$    & Y\\ \hline
		$\{1, 2, 3, 5\}$ & 921600    & 132710400    & 663552000    & 55296000  & $31850496000$    & Y\\ \hline
		$\{1, 2, 3, 6\}$ & 34560000  & 4976640000   & 19906560000  & 829440000 & $1194393600000$  & Y\\ \hline
		$\{1, 2, 4, 5\}$ & 184320    & 53084160     & 53084160     & 184320    & $2548039680$     & Y\\ \hline
		$\{1, 2, 4, 6\}$ & 6912000   & 1990656000   & 1592524800   & 27648000  & $95551488000$    & Y\\ \hline
		$\{1, 2, 5, 6\}$ & 55296000  & 663552000    & 132710400    & 921600    & $31850496000$    & Y\\ \hline
		$\{1, 3, 4, 5\}$ & 2764800   & 1592524800   & 1990656000   & 6912000   & $95551488000$    & Y\\ \hline
		$\{1, 3, 4, 6\}$ & 103680000 & 59719680000  & 59719680000  & 103680000 & $3583180800000$  & Y\\ \hline
		$\{1, 3, 5, 6\}$ & 829440000 & 19906560000  & 4976640000   & 34560000  & $1194393600000$  & Y\\ \hline
		$\{1, 4, 5, 6\}$ & 165888000 & 7962624000   & 398131200    & 115200    & $95551488000$    & Y\\ \hline
		$\{2, 3, 4, 5\}$ & 368640    & 1274019840   & 6370099200   & 22118400  & $305764761600$   & N\\ \hline
		$\{2, 3, 4, 6\}$ & 13824000  & 47775744000  & 191102976000 & 331776000 & $11466178560000$ & N\\ \hline
		$\{2, 3, 5, 6\}$ & 110592000 & 15925248000  & 15925248000  & 110592000 & $3822059520000$  & N\\ \hline
		$\{2, 4, 5, 6\}$ & 22118400  & 6370099200   & 1274019840   & 368640    & $305764761600$   & N\\ \hline
		$\{3, 4, 5, 6\}$ & 331776000 & 191102976000 & 47775744000  & 13824000  & $11466178560000$ & N\\ \hline
		$\{1,2,3,4,5\}$  & 22118400  & 76441190400  & 382205952000 & 1327104000  & $18345885696000$   & Y\\ \hline
		$\{1,2,3,4,6\}$  & 829440000 & 2866544640000 & 11466178560000 & 19906560000 & $687970713600000$  & Y\\ \hline
		$\{1,2,3,5,6\}$ & 6635520000  & 955514880000   & 955514880000   & 6635520000  & $229323571200000$  & Y\\ \hline
		$\{1,2,4,5,6\}$ & 1327104000  & 382205952000   & 76441190400    & 22118400    & $18345885696000$   & Y\\ \hline
		$\{1,3,4,5,6\}$ & 19906560000 & 11466178560000 & 2866544640000  & 829440000   & $687970713600000$  & Y\\ \hline
		$\{2,3,4,5,6\}$ & 2654208000  & 9172942848000  & 9172942848000  & 2654208000  & $2201506283520000$ & N\\ \hline
		$\{1,2,3,4,5,6\}$ & 159252480000 & 550376570880000 & 550376570880000 & 159252480000 & $132090377011200000$ & Y\\ \hline
		\end{tabular}
		\caption{The mix of 4 or more regular convex 4-polytopes}
		\label{tab:4-poly-2}
		\end{sidewaystable}

	\subsection{Ranks 5 and higher}
	
		In ranks 5 and higher, the only regular convex polytopes are the $n$-simplex
		$T^n := \{3^{n-1}\}$, the $n$-cube $B^n := \{4, 3^{n-2}\}$, and the $n$-cross-polytope 
		$C^n := \{3^{n-2}, 4\}$. We will determine the group of their mix and which of the mixes
		are polytopal, as well as the number of faces in each rank.
		
		\begin{theorem}
		\label{thm:convex-dir-prod}
		For $n \geq 3$, we have
		\begin{align*}
		\G^+(T^n \mix B^n) &= \G^+(T^n) \times \G^+(B^n) \\
		\G^+(T^n \mix C^n) &= \G^+(T^n) \times \G^+(C^n) \\
		\G^+(B^n \mix C^n) &= \G^+(B^n) \times \G^+(C^n) \\
		\G^+(T^n \mix B^n \mix C^n) &= \G^+(T^n) \times \G^+(B^n) \times \G^+(C^n).
		\end{align*}
		\end{theorem}
		
		\begin{proof}
		The first three follow immediately from \cref{cor:gcd-comix}. For the last one, we note
		that $T^n \mix B^n$ is of type $\{12, 3^{n-2}\}$ while $C^n$ is of type $\{3^{n-2}, 4\}$,
		so Corollary~\ref{cor:gcd-comix} applies again.
		\end{proof}

		\begin{theorem}
		For $n \geq 4$, $T^n \mix B^n$, $T^n \mix C^n$, and $T^n \mix B^n \mix C^n$ are all polytopal,
		and $B^n \mix C^n$ is not.
		\end{theorem}
		
		\begin{proof}
		Since $T^n$ has the same vertex-figures as $B^n$, and the same
		facets as $C^n$, then the mix with either of these polytopes
		is polytopal by Proposition~\ref{prop:facets-cover}. Now, consider the facets of $T^n \mix B^n$.
		These facets are of type $T^{n-1} \mix B^{n-1}$; in particular, they cover $T^{n-1}$.
		Therefore, the facets of $T^n \mix B^n$ cover the facets of $C^n$, and thus the mix
		$T^n \mix B^n \mix C^n$ is polytopal. For the final case, we note that the facets of
		$B^n \mix C^n$ are $B^{n-1} \mix T^{n-1}$ and that the vertex-figures are $T^{n-1}
		\mix C^{n-1}$. By \tref{thm:convex-dir-prod}, we see that the facets and
		the vertex-figures are both direct products of their components. Then by
		Theorem~\ref{thm:medial-dir-prod}, we see that $B^n \mix C^n$ is polytopal if and
		only if the medial sections also mix as the direct product. But the medial sections
		of $B^n$ and of $C^n$ are both $T^{n-2}$, and thus the mix of the medial sections
		is also $T^{n-2}$. Therefore, $B^n \mix C^n$ is not polytopal.
		\end{proof}

		Now we calculate the number $g$ of flags and the numbers $f_k$ of $k$-faces of each of the mixes. 
		\taref{tab:reg-convex} summarizes this information for $T^n$, $B^n$, and $C^n$.
		
		\begin{table}[htbp]
		\centering
		\begin{tabular}{c c c}
		$\calP$ & $f_k$ & $g$ \\ \hline
		$T^n$ & $\dbinom{n+1}{k+1}$ & $(n+1)!$ \\ \hline
		$B^n$ & $2^{n-k} \dbinom{n}{k}$ & $2^n n!$ \\ \hline
		$C^n$ & $2^{k+1} \dbinom{n}{k+1}$ & $2^n n!$ \\ \hline
		\end{tabular}
		\caption{The regular convex $n$-polytopes for $n \geq 5$}
		\label{tab:reg-convex}
		\end{table}

		Let us calculate the number of flags in each mix. For $n=1$, all of the mixes are
		just segments, with 2 flags. For $n=2$, we get that $T^n \mix B^n = T^n \mix C^n =
		T^n \mix B^n \mix C^n = \{12\}$, with 24 flags, while $B^n \mix C^n = B^n$, with 8 flags.
		Now, for $n \geq 3$, \tref{thm:convex-dir-prod} tells us that the mixes
		are all as the direct product of the components. Then by \cref{cor:flags-mix},
		we see that $T^n \mix B^n$ and $T^n \mix C^n$ both have $2^{n-1} n! (n+1)!$ flags,
		that $B^n \mix C^n$ has $2^{2n-1} (n!)^2$ flags, and that $T^n \mix B^n \mix C^n$
		has $2^{2n-2} (n!)^2 (n+1)!$ flags.

		For any polytope $\calP$, let $g(\calP)$ be the number of flags of $\calP$, and let
		$f_k(\calP)$ be the number of $k$-faces of $\calP$. Then
		\[ g(\calP) = f_k(\calP) g(F_k/F_{-1}) g(F_n/F_k). \]
		We can use this to calculate $f_k(\calP)$. First, in the case $T^n \mix B^n$,
		we see that the $k$-faces are $T^k \mix B^k$, and the co-$k$-faces are $T^{n-1-k} \mix
		T^{n-1-k} = T^{n-1-k}$. Thus we have
		\[ f_k(T^n \mix B^n) = \frac{2^{n-1} n! (n+1)!}{g(T^k \mix B^k) g(T^{n-1-k})}. \]
		Now, for $k=0$, the mix $T^k \mix B^k$ has a single flag; for $k \geq 1$, it has
		$2^{k-1} k! (k+1)!$ flags. So we see that the number of vertices of $T^n \mix B^n$
		is
		\[ f_0(T^n \mix B^n) = \frac{2^{n-1} n! (n+1)!}{(1)(n!)} = 2^{n-1} (n+1)!,\]
		and that the number of $k$-faces for $k \geq 1$ is
		\[ f_k(T^n \mix B^n) = \frac{2^{n-1} n! (n+1)!}{2^{k-1} k! (k+1)! (n-k)!} = 2^{n-k} (n-k)! \binom{n+1}{k+1} \binom{n}{k}.\]
		Since $T^n \mix C^n$ is dual to this, we get that the number of facets is
		$2^{n-1} (n+1)!$, and the number of $k$-faces for $k \leq n-2$ is
		$2^{k+1} (k+1)! \binom{n+1}{k+1} \binom{n}{k+1}$.
		
		Moving on to $B^n \mix C^n$, we have that the $k$-faces are $B^k \mix T^k$, and the co-$k$-faces
		are $T^{n-1-k} \mix C^{n-1-k}$. The number of vertices is equal to the number of facets, which
		is equal to
		\[ \frac{2^{2n-1} (n!)^2}{(1)2^{n-2} (n-1)! n!} = 2^{n+1} n. \]
		The number of $k$-faces for $1 \leq k \leq n-2$ is
		\[ \frac{2^{2n-1} (n!)^2}{2^{k-1} k! (k+1)! 2^{n-k-2} (n-k-1)! (n-k)!} = 2^{n+2} \binom{n}{k} 
		\binom{n}{k+1}.\]
		
		Finally, we consider the mix $T^n \mix B^n \mix C^n$. The $k$-faces are $T^k \mix B^k$ and
		the co-$k$-faces are $T^{n-k-1} \mix C^{n-k-1}$. The number of vertices is equal to the number
		of facets, which is equal to
		\[ \frac{2^{2n-2} (n!)^2 (n+1)!}{(1)2^{n-2} (n-1)! n!} = 2^n n (n+1)!.\]
		The number of $k$-faces for $1 \leq k \leq n-2$ is
		\[ \frac{2^{2n-2} (n!)^2 (n+1)!}{2^{k-1} k! (k+1)! 2^{n-k-2} (n-k-1)! (n-k)!}
		= 2^{n+1} (n+1)! \binom{n}{k} \binom{n}{k+1}. \]
		
		We summarize these results in \taref{tab:mix-reg-convex}.

		\begin{table}[htbp]
		\centering
		\begin{tabular}{c c c c c}
		Mix & $f_0$ & $f_{n-1}$ & $f_k$ & $g$ \\ \hline 
		$T^n \mix B^n$ & $2^{n-1} (n+1)!$ & $2n(n+1)$ & $2^{n-k} (n-k)! \dbinom{n+1}{k+1} \dbinom{n}{k}$ &
		$2^{n-1} n! (n+1)!$ \\ \hline
		$T^n \mix C^n$ & $2n(n+1)$ & $2^{n-1} (n+1)!$ & $2^{k+1} (k+1)! \dbinom{n+1}{k+1} \dbinom{n}{k+1}$ &
		$2^{n-1} n! (n+1)!$ \\ \hline
		$B^n \mix C^n$ & $2^{n+1} n$ & $2^{n+1} n$ & $2^{n+2} \dbinom{n}{k} \dbinom{n}{k+1}$ &
		$2^{2n-1} (n!)^2$ \\ \hline
		$T^n \mix B^n \mix C^n$ & $2^n n (n+1)!$ & $2^n n (n+1)!$ & $2^{n+1} (n+1)! \dbinom{n}{k} \dbinom{n}{k+1}$
		& $2^{2n-2} (n!)^2 (n+1)!$ \\ \hline
		\end{tabular}
		\caption{The mix of the regular convex $n$-polytopes for $n \geq 5$}
		\label{tab:mix-reg-convex}
		\end{table}

\end{document}